\documentclass[10pt]{amsart}
\usepackage{amsmath}
\usepackage{amsxtra}
\usepackage{amscd}
\usepackage{amsthm}
\usepackage{amsfonts}
\usepackage{amssymb}
\usepackage{eucal}
\usepackage[all]{xy}
\usepackage{graphicx}
\usepackage[usenames]{color}

\newtheorem{cor}[subsection]{Corollary}
\newtheorem{lem}[subsection]{Lemma}
\newtheorem{prop}[subsection]{Proposition}

\newtheorem{thm}[subsection]{Theorem}

\theoremstyle{remark}
\newtheorem{rem}[subsection]{Remark}
\newtheorem{example}[subsection]{Example}

\theoremstyle{definition}

\theoremstyle{remark}

\newcommand{\secref}[1]{Sect.~\ref{#1}}
\newcommand{\lemref}[1]{Lemma~\ref{#1}}
\newcommand{\propref}[1]{Proposition~\ref{#1}}
\newcommand{\corref}[1]{Corollary~\ref{#1}}

\numberwithin{equation}{section}

\newcommand{\nc}{\newcommand}
\nc{\renc}{\renewcommand}
\nc{\ssec}{\section}
\nc{\sssec}{\subsubsection}
\nc{\on}{\operatorname}

\newcommand{\epi}{\twoheadrightarrow}

\nc{\ips}{{\iota_P^{(S)}}}
\nc{\ipms}{{\iota_{P^-}^{(S)}}}
\nc{\sfpps}{{\sfp_P^{(S)}}}
\nc{\sfppms}{{\sfp_{P^-}^{(S)}}}

\nc\ol{\overline}
\nc\wt{\widetilde}
\nc\tboxtimes{\wt{\boxtimes}}
\nc\tstar{\wt{\star}}
\nc{\alp}{\alpha}

\nc{\ZZ}{{\mathbb Z}}
\nc{\NN}{{\mathbb N}}
\nc{\OO}{{\mathbb O}}
\renc{\SS}{{\mathbb S}}
\nc{\DD}{{\mathbb D}}
\nc{\GG}{{\mathbb G}}

\nc{\Fq}{{\mathbb F}_q}
\nc{\Fqb}{\ol{{\mathbb F}_q}}
\nc{\Ql}{\ol{{\mathbb Q}_\ell}}
\nc{\id}{\text{id}}
\nc\X{\mathcal X}

\nc{\red}{\on{red}}
\nc{\Ho}{\on{Ho}}
\nc{\Hom}{\on{Hom}}
\nc{\Mor}{\on{Mor}}
\nc{\coef}{\on{coeff}}
\nc{\Lie}{\on{Lie}}
\nc{\Loc}{\on{Loc}}
\nc{\Pic}{\on{Pic}}
\nc{\Bun}{\on{Bun}}
\nc{\IC}{\on{IC}}
\nc{\Aut}{\on{Aut}}
\nc{\rk}{\on{rk}}
\nc{\Sh}{\on{Sh}}
\nc{\Perv}{\on{Perv}}
\nc{\pos}{{\on{pos}}}
\nc{\Conv}{\on{Conv}}
\nc{\Sph}{\on{Sph}}
\nc{\Sym}{\on{Sym}}
\nc{\BunBb}{\overline{\Bun}_B}
\nc{\BunNb}{\overline{\Bun}_N}
\nc{\BunTb}{\overline{\Bun}_T}
\nc{\BunBbm}{\overline{\Bun}_{B^-}}
\nc{\BunBbel}{\overline{\Bun}_{B,el}}
\nc{\BunBbmel}{\overline{\Bun}_{B^-,el}}
\nc{\Buno}{\overset{o}{\Bun}}
\nc{\BunPb}{{\overline{\Bun}_P}}
\nc{\BunBM}{\Bun_{B(M)}}
\nc{\BunBMb}{\overline{\Bun}_{B(M)}}
\nc{\BunPbw}{{\widetilde{\Bun}_P}}
\nc{\BunBP}{\widetilde{\Bun}_{B,P}}
\nc{\GUb}{\overline{G/U}}
\nc{\GUPb}{\overline{G/U(P)}}

\nc{\PP}{\underline{P}'}

\nc{\Hhom}{\underline{\on{Hom}}}
\nc\syminfty{\on{Sym}^{\infty}}
\nc\lal{\ol{\lambda}}
\nc\xl{\ol{x}}
\nc\thl{\ol{\theta}}
\nc\nul{\ol{\nu}}
\nc\mul{\ol{\mu}}
\nc\Sum\Sigma
\nc{\oX}{\overset{o}{X}{}}
\nc{\hl}{\overset{\leftarrow}h{}}
\nc{\hr}{\overset{\rightarrow}h{}}
\nc{\M}{{\mathcal M}}
\nc{\N}{{\mathcal N}}
\nc{\F}{{\mathcal F}}
\nc{\D}{{\mathcal D}}
\nc{\Q}{{\mathcal Q}}
\nc{\Y}{{\mathcal Y}}
\nc{\G}{{\mathcal G}}
\nc{\E}{{\mathcal E}}
\nc{\CalC}{{\mathcal C}}
\nc\Dh{\widehat{\D}}

\nc{\C}{{\mathcal C}}
\nc{\K}{{\mathcal K}}
\renewcommand{\H}{{\mathcal H}}

\nc{\T}{{\mathcal T}}
\nc{\V}{{\mathcal V}}
\renc{\P}{{\mathcal P}}
\nc{\A}{{\mathcal A}}
\nc{\B}{{\mathcal B}}
\nc{\U}{{\mathcal U}}

\nc{\Gr}{{\on{Gr}}}

\nc{\frn}{{\check{\mathfrak u}(P)}}

\nc{\fC}{\mathfrak C}
\nc{\p}{\mathfrak p}
\nc{\q}{\mathfrak q}
\nc\f{{\mathfrak f}}

\nc{\qo}{{\mathfrak q}}
\nc{\po}{{\mathfrak p}}
\nc{\s}{{\mathfrak s}}
\nc\w{\text{w}}

\nc\mathi\iota
\nc\Spec{\on{Spec}}
\nc\Proj{\on{Proj}}
\nc\Mod{\on{Mod}}
\nc{\tw}{\widetilde{\mathfrak t}}
\nc{\pw}{\widetilde{\mathfrak p}}
\nc{\qw}{\widetilde{\mathfrak q}}
\nc{\jw}{\widetilde j}

\nc{\grb}{\overline{\Gr}}
\nc{\I}{\mathcal I}

\nc{\lambdach}{{\check\lambda}}
\nc{\Lambdach}{{\check\Lambda}{}}
\nc{\much}{{\check\mu}}
\nc{\omegach}{{\check\omega}}
\nc{\nuch}{{\check\nu}}
\nc{\etach}{{\check\eta}}
\nc{\alphach}{{\check\alpha}}
\nc{\oblvtach}{{\check\oblvta}}
\nc{\rhoch}{{\check\rho}}
\nc{\ch}{{\check h}}

\nc{\Hb}{\overline{\H}}

\nc{\BA}{{\mathbb{A}}}
\nc{\BC}{{\mathbb{C}}}
\nc{\BG}{{\mathbb{G}}}
\nc{\BM}{{\mathbb{M}}}
\nc{\BO}{{\mathbb{O}}}
\nc{\BD}{{\mathbb{D}}}
\nc{\BBD}{{\mathbf{D}}}
\nc{\BN}{{\mathbb{N}}}
\nc{\BP}{{\mathbb{P}}}
\nc{\BQ}{{\mathbb{Q}}}
\nc{\BR}{{\mathbb{R}}}
\nc{\BZ}{{\mathbb{Z}}}
\nc{\BS}{{\mathbb{S}}}
\nc{\Deep}{{\bf{deep}}}
\nc{\deep}{deep}

\nc{\CA}{{\mathcal{A}}}
\nc{\CB}{{\mathcal{B}}}

\nc{\CE}{{\mathcal{E}}}
\nc{\CF}{{\mathcal{F}}}
\nc{\CH}{{\mathcal{H}}}

\nc{\CL}{{\mathcal{L}}}
\nc{\CC}{{\mathcal{C}}}
\nc{\CG}{{\mathcal{G}}}
\nc{\CalD}{{\mathcal{D}}}
\nc{\CM}{{\mathcal{M}}}
\nc{\CN}{{\mathcal{N}}}
\nc{\CK}{{\mathcal{K}}}
\nc{\CO}{{\mathcal{O}}}
\nc{\CP}{{\mathcal{P}}}
\nc{\CQ}{{\mathcal{Q}}}
\nc{\CR}{{\mathcal{R}}}
\nc{\CS}{{\mathcal{S}}}
\nc{\CT}{{\mathcal{T}}}
\nc{\CU}{{\mathcal{U}}}
\nc{\CV}{{\mathcal{V}}}
\nc{\CW}{{\mathcal{W}}}
\nc{\CX}{{\mathcal{X}}}
\nc{\CY}{{\mathcal{Y}}}
\nc{\CZ}{{\mathcal{Z}}}
\nc{\CI}{{\mathcal{I}}}

\nc{\csM}{{\check{\mathcal A}}{}}
\nc{\oM}{{\overset{\circ}{\mathcal M}}{}}
\nc{\obM}{{\overset{\circ}{\mathbf M}}{}}
\nc{\oCA}{{\overset{\circ}{\mathcal A}}{}}
\nc{\obA}{{\overset{\circ}{\mathbf A}}{}}
\nc{\ooM}{{\overset{\circ}{M}}{}}
\nc{\osM}{{\overset{\circ}{\mathsf M}}{}}
\nc{\vM}{{\overset{\bullet}{\mathcal M}}{}}
\nc{\nM}{{\underset{\bullet}{\mathcal M}}{}}
\nc{\oD}{{\overset{\circ}{\mathcal D}}{}}
\nc{\obD}{{\overset{\circ}{\mathbf D}}{}}
\nc{\oA}{{\overset{\circ}{\mathbb A}}{}}
\nc{\op}{{\overset{\bullet}{\mathbf p}}{}}
\nc{\cp}{{\overset{\circ}{\mathbf p}}{}}
\nc{\oU}{{\overset{\bullet}{\mathcal U}}{}}
\nc{\oZ}{{\overset{\circ}{\mathcal Z}}{}}
\nc{\ofZ}{{\overset{\circ}{\mathfrak Z}}{}}
\nc{\oF}{{\overset{\circ}{\fF}}}

\nc{\fa}{{\mathfrak{a}}}
\nc{\fb}{{\mathfrak{b}}}
\nc{\fd}{{\mathfrak{d}}}
\nc{\ff}{{\mathfrak{f}}}
\nc{\fg}{{\mathfrak{g}}}
\nc{\fgl}{{\mathfrak{gl}}}
\nc{\fh}{{\mathfrak{h}}}
\nc{\fj}{{\mathfrak{j}}}
\nc{\fl}{{\mathfrak{l}}}
\nc{\fm}{{\mathfrak{m}}}
\nc{\fn}{{\mathfrak{n}}}
\nc{\fu}{{\mathfrak{u}}}
\nc{\fp}{{\mathfrak{p}}}
\nc{\fr}{{\mathfrak{r}}}
\nc{\fs}{{\mathfrak{s}}}
\nc{\ft}{{\mathfrak{t}}}
\nc{\fz}{{\mathfrak{z}}}
\nc{\fsl}{{\mathfrak{sl}}}
\nc{\hsl}{{\widehat{\mathfrak{sl}}}}
\nc{\hgl}{{\widehat{\mathfrak{gl}}}}
\nc{\hg}{{\widehat{\mathfrak{g}}}}
\nc{\chg}{{\widehat{\mathfrak{g}}}{}^\vee}
\nc{\hn}{{\widehat{\mathfrak{n}}}}
\nc{\chn}{{\widehat{\mathfrak{n}}}{}^\vee}

\nc{\fA}{{\mathfrak{A}}}
\nc{\fB}{{\mathfrak{B}}}
\nc{\fD}{{\mathfrak{D}}}
\nc{\fE}{{\mathfrak{E}}}
\nc{\fF}{{\mathfrak{F}}}
\nc{\fG}{{\mathfrak{G}}}
\nc{\fK}{{\mathfrak{K}}}
\nc{\fL}{{\mathfrak{L}}}
\nc{\fM}{{\mathfrak{M}}}
\nc{\fN}{{\mathfrak{N}}}
\nc{\fP}{{\mathfrak{P}}}
\nc{\fU}{{\mathfrak{U}}}
\nc{\fV}{{\mathfrak{V}}}
\nc{\fZ}{{\mathfrak{Z}}}

\nc{\bb}{{\mathbf{b}}}
\nc{\bc}{{\mathbf{c}}}
\nc{\bd}{{\mathbf{d}}}
\nc{\bbf}{{\mathbf{f}}}
\nc{\be}{{\mathbf{e}}}
\nc{\bi}{{\mathbf{i}}}
\nc{\bj}{{\mathbf{j}}}
\nc{\bn}{{\mathbf{n}}}
\nc{\bp}{{\mathbf{p}}}
\nc{\bq}{{\mathbf{q}}}
\nc{\bu}{{\mathbf{u}}}
\nc{\bv}{{\mathbf{v}}}
\nc{\bx}{{\mathbf{x}}}
\nc{\bs}{{\mathbf{s}}}
\nc{\by}{{\mathbf{y}}}
\nc{\bw}{{\mathbf{w}}}
\nc{\bA}{{\mathbf{A}}}
\nc{\bK}{{\mathbf{K}}}
\nc{\bB}{{\mathbf{B}}}
\nc{\bC}{{\mathbf{C}}}
\nc{\bG}{{\mathbf{G}}}
\nc{\bD}{{\mathbf{D}}}
\nc{\bH}{{\mathbf{H}}}
\nc{\bM}{{\mathbf{M}}}
\nc{\bN}{{\mathbf{N}}}
\nc{\bV}{{\mathbf{V}}}
\nc{\bW}{{\mathbf{W}}}
\nc{\bX}{{\mathbf{X}}}
\nc{\bZ}{{\mathbf{Z}}}
\nc{\bS}{{\mathbf{S}}}

\nc{\sA}{{\mathsf{A}}}
\nc{\sB}{{\mathsf{B}}}
\nc{\sC}{{\mathsf{C}}}
\nc{\sD}{{\mathsf{D}}}
\nc{\sF}{{\mathsf{F}}}
\nc{\sG}{{\mathsf{G}}}
\nc{\sK}{{\mathsf{K}}}
\nc{\sM}{{\mathsf{M}}}
\nc{\sO}{{\mathsf{O}}}
\nc{\sW}{{\mathsf{W}}}
\nc{\sQ}{{\mathsf{Q}}}
\nc{\sP}{{\mathsf{P}}}
\nc{\sZ}{{\mathsf{Z}}}
\nc{\sfp}{{\mathsf{p}}}
\nc{\bsfp}{{\mathsf{\bar p}_P}}
\nc{\sfq}{{\mathsf{q}}}
\nc{\sr}{{\mathsf{r}}}
\nc{\bk}{{\mathsf{k}}}
\nc{\sg}{{\mathsf{g}}}
\nc{\sff}{{\mathsf{f}}}
\nc{\sfb}{{\mathsf{b}}}
\nc{\sfc}{{\mathsf{c}}}
\nc{\sd}{{\mathsf{d}}}

\nc{\BK}{{\bar{K}}}

\nc{\tA}{{\widetilde{\mathbf{A}}}}
\nc{\tB}{{\widetilde{\mathcal{B}}}}
\nc{\tg}{{\widetilde{\mathfrak{g}}}}
\nc{\tG}{{\widetilde{G}}}
\nc{\TM}{{\widetilde{\mathbb{M}}}{}}
\nc{\tO}{{\widetilde{\mathsf{O}}}{}}
\nc{\tU}{{\widetilde{\mathfrak{U}}}{}}
\nc{\TZ}{{\tilde{Z}}}
\nc{\tx}{{\tilde{x}}}
\nc{\tbv}{{\tilde{\bv}}}
\nc{\tfP}{{\widetilde{\mathfrak{P}}}{}}
\nc{\tz}{{\tilde{\zeta}}}
\nc{\tmu}{{\tilde{\mu}}}

\nc{\urho}{\underline{\rho}}
\nc{\uB}{\underline{B}}
\nc{\uC}{{\underline{\mathbb{C}}}}
\nc{\ui}{\underline{i}}
\nc{\uj}{\underline{j}}
\nc{\ofP}{{\overline{\mathfrak{P}}}}
\nc{\oB}{{\overline{\mathcal{B}}}}
\nc{\og}{{\overline{\mathfrak{g}}}}
\nc{\oI}{{\overline{I}}}

\nc{\eps}{\varepsilon}
\nc{\hrho}{{\hat{\rho}}}

\nc{\one}{{\mathbf{1}}}
\nc{\two}{{\mathbf{t}}}

\nc{\Rep}{{\mathop{\operatorname{\rm Rep}}}}
\nc{\Tot}{{\mathop{\operatorname{\rm Tot}}}}
\nc{\Ker}{{\mathop{\operatorname{\rm Ker}}}}
\nc{\im}{{\mathop{\operatorname{\rm Im}}}}
\nc{\Hilb}{{\mathop{\operatorname{\rm Hilb}}}}
\nc{\End}{{\mathop{\operatorname{\rm End}}}}
\nc{\Ext}{{\mathop{\operatorname{\rm Ext}}}}
\nc{\CHom}{{\mathop{\operatorname{{\mathcal{H}}\it om}}}}
\nc{\GL}{{\mathop{\operatorname{\rm GL}}}}
\nc{\gr}{{\mathop{\operatorname{\rm gr}}}}
\nc{\HN}{{\mathop{\operatorname{\rm HN}}}}
\nc{\Id}{{\mathop{\operatorname{\rm Id}}}}
\nc{\de}{{\mathop{\operatorname{\rm def}}}}
\nc{\length}{{\mathop{\operatorname{\rm length}}}}
\nc{\supp}{{\mathop{\operatorname{\rm supp}}}}

\nc{\Cliff}{{\mathsf{Cliff}}}
\nc{\Fl}{\on{Fl}}
\nc{\Fib}{{\mathsf{Fib}}}
\nc{\Coh}{{\on{Coh}}}
\nc{\QCoh}{{\on{QCoh}}}
\nc{\IndCoh}{{\on{IndCoh}}}
\nc{\FCoh}{{\mathsf{FCoh}}}

\nc{\reg}{{\text{\rm reg}}}

\nc{\cplus}{{\mathbf{C}_+}}
\nc{\cminus}{{\mathbf{C}_-}}
\nc{\cthree}{{\mathbf{C}_*}}
\nc{\Qbar}{{\bar{Q}}}
\nc\Eis{\on{Eis}}
\nc\Eisb{\ol\Eis{}}
\nc\Eisr{\on{Eis}^{rat}{}}
\nc\wh{\widehat}
\nc{\Def}{\on{Def_{\check{\fb}}(E)}}
\nc{\barZ}{\overline{Z}{}}
\nc{\barbarZ}{\overline{\barZ}{}}
\nc{\barpi}{\overline\pi}
\nc{\barbarpi}{\overline\barpi}
\nc{\barpip}{\overline\pi{}^+}
\nc{\barpim}{\overline\pi{}^-}

\nc{\fq}{\mathfrak q}

\nc{\fqb}{\ol{\sfq}{}}
\nc{\fpb}{\ol{\sfp}{}}
\nc{\fpr}{{\sfp^{rat}}{}}
\nc{\fqr}{{\sfq^{rat}}{}}

\nc{\hattimes}{\wh\otimes}

\nc{\bh}{{\bar{h}}}
\nc{\bOmega}{{\overline{\Omega(\check \fn)}}}

\nc{\seq}[1]{\stackrel{#1}{\sim}}

\nc{\cT}{{\check{T}}}
\nc{\cG}{{\check{G}}}
\nc{\cM}{{\check{M}}}
\nc{\cB}{{\check{B}}}

\nc{\ct}{{\check{\mathfrak t}}}
\nc{\cg}{{\check{\fg}}}
\nc{\cb}{{\check{\fb}}}
\nc{\cn}{{\check{\fn}}}

\nc{\cLambda}{{\check\Lambda}}

\nc{\cla}{{\check\lambda}}
\nc{\cmu}{{\check\mu}}
\nc{\cnu}{{\check\nu}}
\nc{\ceta}{{\check\eta}}

\nc{\DefbE}{{\on{Def}_{\cB}(E_\cT)}}

\nc{\imathb}{{\ol{\imath}}}
\nc{\rlr}{\overset{\longrightarrow}{\underset{\longrightarrow}\longleftarrow}}

\nc{\oBun}{\overset{\circ}\Bun}
\nc{\LocSys}{\on{LocSys}}
\nc{\BunBbb}{\ol{\ol{Bun}}_B}
\nc{\BunBr}{\Bun_B^{rat}}
\nc{\BunBrsg}{\Bun_B^{rat,\on{s.g.}}}
\nc{\BunBrp}{\Bun_B^{rat,polar}}
\nc{\BunBrpbg}{\Bun_B^{rat,polar,\on{b.g.}}}
\nc{\BunBrpsg}{\Bun_B^{rat,polar,\on{s.g.}}}
\nc{\BunTrp}{\Bun_T^{rat,polar}}
\nc{\BunTrpbg}{\Bun_T^{rat,polar,\on{b.g.}}}
\nc{\BunTrpsg}{\Bun_T^{rat,polar,\on{s.g.}}}
\nc{\BunNr}{\Bun_N^{rat}}
\nc{\BunNre}{\Bun_N^{enh,rat}}
\nc{\BunTr}{\Bun_T^{rat}}
\nc{\Vect}{\on{Vect}}
\nc{\Whit}{\on{Whit}}
\nc{\CTb}{\ol{\on{CT}}}
\nc{\Ran}{\on{Ran}}
\nc{\CTr}{\on{CT}^{rat}{}}
\nc\jmathr{\jmath^{rat}{}}
\nc{\ux}{\underline{x}}
\nc{\clambda}{{\check\lambda}}
\nc{\calpha}{{\check\alpha}}
\nc{\ind}{{\mathbf{ind}}}
\nc{\oblv}{{\mathbf{oblv}}}
\nc{\ox}{{\overline{x}}}
\nc{\cLa}{\check{\Lambda}}
\nc{\StinftyCat}{\on{DGCat}}
\nc{\inftyCat}{\infty\on{-Cat}}
\nc{\inftygroup}{\infty\on{-Grpd}}
\nc{\Dmod}{\on{D-mod}}
\nc{\CMaps}{{\mathcal Maps}}
\nc{\Maps}{\on{Maps}}
\nc{\affSch}{\on{Sch}^{\on{aff}}}
\nc{\dr}{{\on{dR}}}
\nc{\rD}{{\blacktriangle}}
\nc{\oCY}{\overset{\circ}\CY}
\nc{\leqG}{\underset{G}\leq}
\nc{\leqM}{\underset{M}\leq}
\nc{\leqGad}{\underset{G_{ad}}\leq}
\nc{\leqMad}{\underset{M_{ad}}\leq}
\nc{\psId}{\on{Ps-Id}}

\nc{\sotimes}{\overset{!}\otimes}

\begin{document}

\title[On the Langlands retraction]{On the Langlands retraction}

\author{V.~Drinfeld}
\address{Vladimir Drinfeld, Department of Mathematics, University of Chicago.}
\email{drinfeld@math.uchicago.edu}


\begin{abstract}

Given a root system in a vector space $V$,  Langlands defined in 1973 a canonical retraction 
$\fL :V\epi V^+$, where $V^+\subset V$ is the dominant chamber.
In this note we give a short review of the  material on this retraction
(which is well known under the name of ``Langlands' geometric lemmas").

The main purpose of this review is to provide a convenient reference for the work \cite{DrGa}, in which the Langlands retraction is used to define a coarsening of the Harder-Narasimhan-Shatz stratification of the stack of $G$-bundles on a smooth projective curve.

\end{abstract}

\maketitle


\ssec{Introduction}  

Given a root system in a Euclidean space $V$,  Langlands defined in \cite[Sect.~4]{La} a certain retraction 
$\fL :V\to V^+$, where $V^+$ is the dominant chamber. Later this retraction was discussed in 
 \cite[Ch. IV, Subsect.~3.3]{BoWa} and \cite[Sect.~1]{C}. 

In this note we briefly recall the definition and properties of $\fL$.
It has no new results compared with \cite{La} and  \cite{C}; my goal is only to provide a convenient reference
for the work \cite{DrGa} and possibly for some future works.

\medskip

Following J.~Carmona, we
begin in \secref{ss:retraction_metric} with the most naive definition of $\fL$ (which makes sense for a Euclidean space equipped with
\emph{any} basis $\{\alpha_i \}$): namely, $\fL (x)$ is the point of $V^+$ closest to $x$. 

Starting with Section~\ref{ss:key_statements}, we assume that 
$\langle \alpha_i\,, \alpha_j \rangle\le 0$ for $i\ne j$. The key point is that
under this assumption $\fL$ can be characterized 
in terms of the usual ordering on~$V$: namely, Corollary~\ref{c:Langlands} says that $\fL (x)$ is the
least element of the set 
\begin{equation}   \label{e:ge}
\{ y\in V^+\,|\, y\ge x\}.
\end{equation}

It is this characterization of $\fL$ that is important for most applications  (in particular, it is used in 
\cite[Appendix B]{DrGa}). One can consider it as a definition of $\fL$ and
Corollary~\ref{c:Langlands} as a way to prove the existence of the least element of the set \eqref{e:ge}.
In Section~\ref{ss:another} we give \emph{another} proof of this fact, which is independent of Sections~\ref{ss:retraction_metric}-\ref{ss:key_statements}; closely related to it are Remark~\ref{r:concave} and Example~\ref{ex:concave}.



In Section~\ref{ss:reductive groups} we define the Langlands retraction as a map from the space of rational coweights of a reductive group to the dominant cone.

In Section~\ref{ss:history} we make some historical remarks.

\bigskip

I thank R.~Bezrukavnikov and R.~Kottwitz for
drawing my attention to Langlands' articles  \cite{La65,La}. I thank S.~Schieder and A.~Zelevinsky for valuable comments. The author's research was partially supported by NSF grant DMS-1001660.


\ssec{The retraction defined by the metric}    \label{ss:retraction_metric}

Let $V$ be a finite-dimensional vector space over $\BR$ with a positive definite scalar
product $\langle \; ,\; \rangle$.
Let $\{\alpha_i \}_{i \in\Gamma}$ be an arbitrary basis in $V$ and $\{\omega_i \}_{i \in\Gamma}$ the dual basis.
Let $V^+\subset V$ denote the closed convex cone generated by the $\omega_i$'s, $i \in\Gamma$. 

Following J.~Carmona \cite[Sect.~1]{C}, we define the \emph{Langlands retraction} $\fL :V\to V^+$ as follows: $\fL (x)$ is the point of $V^+$ closest to $x$ (such point exists and is unique because
$V^+$ is closed and convex). It is easy to see that the map $\fL$ is continuous.

Let us give another description of $\fL$. 
For a subset $J\subset\Gamma$ let $K_J$ denote the closed convex cone generated by $\omega_j$ for $j \in\Gamma - J$ and by $- \alpha_i$ for $i \in  J$. Clearly, each $K_J$ is a simplicial cone of full dimension in $V$. Let $V_J$ denote the linear span of
$\alpha_j$, $j\in J$ (so $V_J^{\perp}$ is spanned by $\omega_i$, $i\not\in J$).
Let $\on{pr}_J :V\to V$ denote the orthogonal projection onto $V_J^{\perp}$,
so $\on{ker}\,(\on{pr}_J)=V_J$.

\begin{prop}   \label{p:Zel2}
\noindent{\em(a)} The map $\fL$ is piecewise linear. The cones $K_J$ 
are exactly the linearity domains of $\fL$. For $x\in K_J$ one has $\fL (x)=\on{pr}_J (x)$.

\noindent{\em(b)} The cones $K_J$ and their faces form a complete simplicial fan\footnote{This means that these cones cover $V$ and 
each intersection  $K_J\cap K_{J'}$ is a face in both $K_J$ and $K_{J'}\,$.} in
$V$, combinatorially equivalent to the coordinate fan\footnote{The coordinate fan is what one gets when
the basis $\{ \alpha_i\}$ is orthogonal.}.
\end{prop}

\begin{rem}
The wording in the above proposition was suggested to us by A.~Zelevinsky.
\end{rem}

The proposition immediately follows from the next lemma, whose proof is straightforward.


\begin{lem}   \label{l:Zel1}
Let $x \in V$ and $y\in V^+$. Set $J:=\{ j\in\Gamma\,|\,\langle\alpha_j,y\rangle =0\}$.
Then the following are equivalent:

{\em(a)} $y = \fL(x)$.

{\em(b)} $x-y$ belongs to the 
the closed convex cone generated by $- \alpha_j$ for $j \in J$. \qed
\end{lem}

%
%
%

\ssec{The key statements}   \label{ss:key_statements}
Let $V^{pos}$ denote the cone dual to $V^+$, i.e., the closed convex cone generated by the 
$\alpha_i$'s, $i \in\Gamma$. Equip $V$ and $V^+$ with the following partial ordering:
$x\le y$ if $y-x\in V^{pos}$.
By Lemma~\ref{l:Zel1}, the retraction $\fL :V\to V^+$ from Section~\ref{ss:retraction_metric} has the
following property:

\begin{equation}    \label{e:Lxgex}
\fL (x)\ge x, \quad x\in V.
\end{equation}

\begin{thm}    \label{t:Langlands} 
Assume that 
\begin{equation}     \label{e:obtuse}
\langle \alpha_i\,, \alpha_j \rangle\le 0 \mbox{ for } i\ne j. 
\end{equation}
Then the retraction $\fL :V\to V^+$ is order-preserving.
\end{thm} 

By \eqref{e:Lxgex}, Theorem~\ref{t:Langlands} implies the following statement, which
characterizes $\fL$ in terms of the order relation. 

\begin{cor}      \label{c:Langlands} 
If \eqref{e:obtuse} holds then $\fL (x)$ is the least element in $\{ y\in V^+\,|\, y\ge x\}$. \qed
\end{cor}

Let us prove Theorem~\ref{t:Langlands}. To show that a piecewise linear map is order-preserving
it suffices to check that this is true on each of its linearity domains. So Theorem~\ref{t:Langlands} follows from Proposition~\ref{p:Zel2}(a) and the next proposition, which I learned from S.~Schieder \cite[Prop.3.1.2(a)]{Sch}.

\begin{prop}     \label{l:order-preservation} 
Assume \eqref{e:obtuse}. Then for each subset $J\subset\Gamma$ the map $\on{pr}_J :V\to V$ defined in Section~\ref{ss:retraction_metric} is order-preserving.
\end{prop}

To prove the proposition, we need the following lemma.

\begin{lem}   \label{l:well_known}
Let $J\subset\Gamma$. Suppose that $x\in V_J$ and $\langle x\,,\alpha_j\rangle\ge 0$ for all $j\in J$. Then $x\ge 0$.
\end{lem}

\begin{proof}[Proof of the lemma]
We can assume that $J=\Gamma$ (otherwise replace $V$ by $V_J$ and $\Gamma$ by 
$\Gamma_J$). Then the lemma just says that $V^+\subset V^{pos}$. This is a well known consequence of \eqref{e:obtuse}.
\end{proof}

\begin{proof}[Proof of Proposition~\ref{l:order-preservation}]
We have to show that 
$\on{pr}_J (\alpha_i )\ge 0$ for any $i\in\Gamma$. If $i\in J$ then $\on{pr}_J (\alpha_i )= 0$. Now suppose that $i\not\in J$. By the definition of $\on{pr}_J\,$, we have 
$\on{pr}_J (\alpha_i )=\alpha_i+x$, where $x$ is the element of $V_J$ such that 
$\langle x\,,\alpha_j\rangle=-\langle \alpha_i\,,\alpha_j\rangle$ for all 
$j\in J$. By \eqref{e:obtuse} and Lemma~\ref{l:well_known}, $x\ge 0$, so 
$\on{pr}_J (\alpha_i )=\alpha_i+x\ge 0$. 
\end{proof}

\ssec{Another approach to the Langlands retraction}   \label{ss:another}
Suppose that \eqref{e:obtuse} holds. Then one could take \corref{c:Langlands} as the \emph{definition} of the Langlands retraction $\fL :V\to V^+$, i.e., one could define $\fL (x)$ to be the least element of the set 
$\{ y\in V^+\,|\, y\ge x\}$. This set is closed and non-empty (because \eqref{e:obtuse} implies that 
$V^+\subset V^{pos}$), so the existence of the least element in it follows from the next proposition.

\begin{prop}    \label{p:infinum_belongs}
Suppose that $\langle \alpha_i\,, \alpha_j \rangle\le 0$  for $i\ne j$. Then the infinum of any non-empty subset of $V^+$ belongs to $V^+$.
\end{prop}

Here ``infinum" is understood in terms of the partial ordering defined by $V^{pos}$. In other words, given a family of vectors
\begin{equation}   \label{e:given_vectors}
x_t\in V, \quad\quad x_t=\sum_i x_{i,t}\cdot\alpha_i \, ,
\end{equation}
its infinum equals $\sum\limits_i y_i\cdot\alpha_i$, where $y_i:=\inf\limits_t  x_{i,t}\,$. Note that if $x_t\in V^+$ then $x_t\in V^{pos}$, so $x_{i,t}\ge 0$ and  $\inf\limits_t  x_{i,t}$ exists.

\begin{proof}[Proof of \propref{p:infinum_belongs}]
Suppose that we have a family of vectors $x_t\in V^+$ and $y=\inf\limits_t  x_t$. The assumption $x_t\in V^+$ means that $\langle x_t\,, \alpha_i \rangle\ge 0$ for all $i$. We have to show that 
$\langle y\,, \alpha_i \rangle\ge 0$ for all $i$.

Fix $i$. Write $x_t=x'_t+x''_t$, $y=y'+y''$, where $$x'_t,y'\in\BR\alpha_i, \quad
x''_t,y''\in\bigoplus\limits_{j\ne i}\BR\alpha_j\,.$$
Clearly $y'=\inf\limits_t x'_t$, $y''=\inf\limits_t x''_t$. Then for every $t$ one has
\[
\langle x'_t\,, \alpha_i \rangle=\langle x_t\,, \alpha_i \rangle - \langle x''_t\,, \alpha_i \rangle
\ge -\langle x''_t\,, \alpha_i \rangle\ge -\langle y''\,, \alpha_i \rangle
\]
(the second inequality holds because $-\langle\alpha_j\,, \alpha_i \rangle\ge 0$ for $j\ne i$). So
\[
 \langle y'\,, \alpha_i \rangle=\inf\limits_t \langle x'_t\,, \alpha_i \rangle\ge - \langle y''\,, \alpha_i \rangle,
 \]
i.e., $\langle y\,, \alpha_i \rangle\ge 0$.
\end{proof}

\begin{rem}    \label{r:concave}
In the situation of the following example \propref{p:infinum_belongs} just says that the infinum of any family of concave functions is concave. In fact, the above proof of \propref{p:infinum_belongs} is identical to the proof of this classical statement.
\end{rem}

\begin{example}   \label{ex:concave}
Consider the root system of $SL(n)$. In this case $V$ is the orthogonal complement of  the vector 
$\varepsilon_1+\ldots+\varepsilon_n$ in the Euclidean space with orthonormal basis 
$\varepsilon_1,\ldots\varepsilon_n$, and $\alpha_i=\varepsilon_i-\varepsilon_{i+1}$, $1\le i\le n-1$.
Let $\omega_i\in V$ be the basis dual to $\alpha_i$. For each $v\in V$ define $f_{v}:\{ 0,\ldots ,n\}\to\BR$ by
\[
f_{v}(0)= f_{v}(n)=0,\quad  f_{v}(i)=\langle v\, ,\omega_i \rangle
\quad\mbox{ for } 0<i<n.
\]
Then the map $v\mapsto f_{v}$ identifies $V$ with the space of functions $f:\{ 0,\ldots ,n\}\to\BR$ such that 
$f (0)= f(n)=0$. Moreover, 
$V^{pos}$ identifies with the subset of non-negative functions $f$ and $V^+$ with the subset of \emph{concave} functions $f$. Thus the Langlands retraction assigns to a function $f:\{ 0,\ldots ,n\}\to\BR$ 
the smallest concave function which is $\ge f$. 
\end{example}



\ssec{Reductive groups}  \label{ss:reductive groups}
\subsection{A remark on rationaility}
Suppose that in the situation of \secref{ss:retraction_metric} one has $\langle \alpha_i\,,\alpha_j\rangle\in\BQ$ for all $i,j\in\Gamma$.
Then the $\BQ$-linear span of the $\alpha_i's$ equals the $\BQ$-linear span of the $\omega_i's$.
Denote it by $V^\BQ$. Then $V=V^\BQ\otimes\BR$. The cones $K_J$, the subspaces $V_J$, and the operators $\on{pr}_J$ from Section~\ref{ss:retraction_metric} are clearly defined over $\BQ$. So by Proposition~\ref{p:Zel2}, one has
\begin{equation}    \label{e:rationality}
\fL (V^\BQ)\subset V^\BQ .
\end{equation}


\subsection{The Langalnds retraction for coweights}

Now let $G$ be a connected reductive group over an algebraically closed field. Let $\Lambda_G$ be its coweight lattice, i.e., 
$\Lambda_G=\Hom (\BG_m,T)$, where $T$ is \emph{the} maximal torus of $G$. 
Set $\Lambda^{\BQ}_G:=\Lambda_G\otimes\BQ$. We have the simple coroots 
$\check\alpha_i\in \Lambda_G$ and the simple roots $\alpha_i\in \Hom (\Lambda_G,\BZ )$.
Let $\Lambda^{+,\BQ}_G\subset\Lambda^{\BQ}_G$ denote the dominant cone. 
Equip $\Lambda^{+,\BQ}_G$ with the following partial ordering: $\lambda_1\leqG\lambda_2$ if
$\lambda_2-\lambda_1$ is a linear combination of the simple coroots  with non-negative coefficients.

Now define the \emph{Langlands retraction} $\fL_G:\Lambda^{\BQ}_G\to\Lambda^{+,\BQ}_G$  as follows:
$\fL_G (\lambda )$ is the least element of the set
\begin{equation}    \label{e:ge_lambda_in_the chamber}
\{\mu\in\Lambda^{+,\BQ}_G\;|\; \mu\underset{G}\ge\lambda \}
\end{equation}
with respect to the $\leqG$ ordering.

\begin{cor}      \label{c:reductive_groups}
(i) $\fL_G (\lambda )$ exists.

(ii) $\fL_G (\lambda )$ is the element of $\Lambda^{+,\BQ}_G$ closest to $\lambda$ with respect to any positive scalar product on $\Lambda^{+,\BQ}_G\otimes\BR$ which is  invariant with respect to the Weyl group.

(iii) $\fL_G (\lambda )$ is the unique element of the set \eqref{e:ge_lambda_in_the chamber} 
with the following property:
$\langle \fL_G (\lambda )\, ,\alpha_i\rangle=0$ 
for any simple root $\alpha_i$ such that
the coefficient of $\check\alpha_i$ in $\fL_G (\lambda )-\lambda$ is nonzero.
\end{cor}

\begin{proof}
Combine  \lemref{l:Zel1},  \corref{c:Langlands}, and the inclusion \eqref{e:rationality}.
\end{proof}

\subsection{Example: $G=GL(n)$}
In this case, just as in Example~\ref{ex:concave}, one identifies $\Lambda^{\BQ}_G$ with the space of functions $f:\{ 0,\ldots ,n\}\to\BQ$ such that $f(0)=0$ (while $f(n)$ is arbitrary). Then the subset $\Lambda^{\BQ}_G\subset\Lambda^{+,\BQ}_G$ identifies with the subset of \emph{concave} functions $f:\{ 0,\ldots ,n\}\to\BQ$ with $f(0)=0$. Just as in Example~\ref{ex:concave}, the Langlands retraction assigns to a function 
$f:\{ 0,\ldots ,n\}\to\BQ$ the smallest concave function which is $\ge f$. 

\ssec{Some historical remarks}    \label{ss:history}
In \cite{La} R.~Langlands defined the retraction $\fL$ and formulated his ``geometric lemmas" 
(see \cite[Lemmas 4.4-4.5 and Corollary 4.6]{La}) for the purpose of 
the classification of representations of real reductive groups in terms of tempered ones. However,
much earlier he had formulated a closely related (and more complicated) combinatorial lemma\footnote{An elementary introduction to this lemma can be found in \cite{Cas1,Cas2}.} in his theory of Eisenstein series, see \cite[Sect.~8]{La65}. In this work Langlands considers Eisenstein series on quotients of the form $G(\BR )/\Gamma$, where $G$ is a reductive group over $\BQ$ and $\Gamma$ is an arithmetic subgroup, but the same technique applies to quotients of the form $G(\BA)/G (\BQ )$. Note that the stack $\Bun_G$ considered in \cite{DrGa} is not far away from $G(\BA)/G (\BQ )$, so the fact that the Langlands retraction is used in 
\cite[Appendix B]{DrGa} is not surprising.


\end{document}